\documentclass[a4 article]{article}

 \usepackage{amssymb,amsthm,amsmath}

\def\a{\alpha}

\def\lam{\lambda}

\renewcommand{\S}{\mathbb{S}}

\newcommand{\T}{\mathbb{T}}

\newcommand{\C}{\mathbb{C}}

\newcommand{\R}{\mathbb{R}}

\newcommand{\Z}{\mathbb{Z}}

 \newcommand{\N}{\mathbb{N}}

  \def\cH{{\cal H}}  

   \def\cO{{\cal O}} 

   \def\cP{{\cal P}}

\newcommand{\eps}{\varepsilon}

\renewcommand{\leq}{\leqslant}

\renewcommand{\geq}{\geqslant}

\renewcommand{\epsilon}{\varepsilon}

\usepackage[english]{babel}

\newtheorem{theo}{Theorem}

\newtheorem{lemm}{Lemma}

\newtheorem{prop}{Proposition}

\newtheorem{rema}{Remark}

\newtheorem{question}{Question}

\newtheorem{defin}{Definition}

\title{Analytic uniquely ergodic volume preserving maps on odd spheres}
\author{Bassam Fayad and Anatole Katok\footnote{Based on research supported by the NSF grant 1002554}}

\begin{document}

\maketitle
\begin{abstract} We construct examples of volume-preserving uniquely ergodic (and hence minimal) real-analytic diffeomorphisms on odd-dimemsional spheres.
\end{abstract} 

\section{Introduction} 

There is a general belief that topology of a manifold $M$ with some  low-dimensional exceptions,   does not influence ergodic properties of volume-preserving dynamical systems on $M$ and that  restriction on  topological properties  of systems with strong recurrence, say, come only  from algebraic and differential topology rather than from dynamics. 

There are two aspects here: (i) the {\em smooth realization problem}  that asks what  isomorphism types or properties of measure-preserving transformations or flows  appear for  volume-preserving
dynamical systems on  a compact manifold and (ii) the {\em phase space dependence}: given  an  isomorphism type or property (measurable or topological)  that appear in a  smooth dynamical system on a compact manifold $M$
describe  the class of manifolds that allow a system of the same kind.

We do not discuss the smooth realization problem here.   It is enough to mention that, while the only known restriction is finiteness of entropy (and it is not specific to  systems preserving a smooth measure and true for any Borel measure),  very few systems 
 that are naturally not smooth  have been shown to allow a smooth realization, e.g. certain translation on the infinite-dimensional tori, see \cite{AK},   and certain unpublished constructions. 
 
 More is known  about the phase space dependence. For example, using  a  surjective continuous map diffeomorphic on the interior from the closed  disc onto an arbitrary  compact manifold (closed or with boundary) of the same dimension  one shows that every system that can be realized on the disc $\mathbb D^n$ an is sufficiently flat  at the boundary can be  realized on an arbitrary $n$-dimensional manifold. This was used to show existence of a  zero entropy ergodic diffeomorphisms and flows \cite{A} and 
 Bernoulli transformations \cite{BFK}.

 Beyond Bernoulli  case and its simple concatenations with zero entropy examples,   smooth positive entropy examples are few and far between, see  \cite{K-inv, Rudolph}.  In the zero entropy setting
 however there is  versatile {\em approximation by conjugation}  method originally introduced in \cite{AK} and sometimes called Anosov-Katok method. We use this method in the present work and introduce  basic scheme in Section~\ref{ACscheme}. For a detailed modern overview of the method see \cite{FK}. In the discrete time case as the starting point the method requires an  effective smooth action of the  circle 
 (for ergodic properties)  or free or at least locally free action of the circle (for properties  involving behavior of all orbits such as minimality or  unique ergodicity). Similarly for the continuous time constructions   an action of the two-dimenssional torus  on the ambient manifold is needed. 
 
After a long lull following  original  development in the late 1960s -- mid 1970s this methods enjoyed a lively  resurrection  during the last decade or so. As examples of important  advances during that period  one  should mention  a multiple frequency'' version of the method that allows to produce  new classes of mixing examples \cite{Fayad-mixing, FK} and     realization of any circle  rotation  with a Liouvillean rotation number as a  $C^\infty$ volume preserving diffeomorphism 
 of any  compact manifold with a circle action  \cite{FSW}. 
 
 Applicability of the approximation by conjugation method  critically depends on construction of successive conjugating diffeomorphisms with prescribed behavior. This conjugacies are invariably  very large in the appropriate topologies but they should lie in the space; e.g. all derivatives for the map and its inverse must exist although they may  be very large. What is required from those conjugacies 
 is controlled behavior in a large mart of the phase space. In the smooth category such constructions are readily available since maps  defined on various parts of the space can be glued together.\footnote{Notice however difficulties of the global character that appear in the symplectic versions of the method, see \cite{K-Hamiltonian}.} However the situation changes drastically when one passes to the real-analytic category. The most basic property  required to start  the construction in a particular  class  is transitivity 
 of the action by  diffeomorphisms of that class on pairs  of points.  In the  setting of a real-analytic manifold $M$ this 
 means existence of  diffeomorphisms $H_{x,y}$   for any pair $x,y\in M$  such that $H_{x,y}x=y$
 such that both $H_{x,y}$ and their inverses extend to a fixed complex neighborhood of $M$. 
 We are not aware of such a fact for close manifolds  but for manifolds with boundaries  or  for a restricted version, say   requiring that $H$ fixes a point $z$  that is excluded from the construction
 there are obvious difficulties.  Those are  situations that appear  for example in the most basic cases  where effective analytic action of the circle exists: the disc $\mathbb D^2$ and the two-dimensional sphere $S^2$.  Accordingly the following basic question is still open: 
 
 \begin{question} Does there exist a real analytic area preserving diffeomorphism of $\mathbb D^2$ or  $\S^2$ that is ergodic and  has zero entropy?
\end{question} 

Other properties such as ergodicity and closeness to the identity, almost minimality, almost unique ergodicity, etc. are not available on the disc or the sphere in the real analytic category.\footnote{The original Bernoulli construction on the sphere or the disc from \cite{K-79} can be   carried out in the real-analytic category with proper adjustments; see \cite{Gerber, LdeSa}.}

 In this paper we consider the  most basic  situation  where such a transitive family commuting with a 
 free action of the circle is  present, namely the odd-dimensional spheres. A construction of volume-preserving uniquely ergodic  real-analytic diffeomorphisms   on $\S^{2n+1}$  was outlines in \cite{FK}. 
 In the present paper we give complete proofs. 
 
Let us emphasize that spheres are considered to present the method in a succinct  way. Existence of a transitive family with large domain of analyticity is the key. For example, our results extend fairly straightforwardly to the  case of compact Lie groups, the setting is explained  in Section~\ref{results}.
An even more general setting is possible; it will appear in a subsequent paper.
 
\section{Formulation of the result and outline of proof}





\subsection{Notations.} We will consider the standard embedding of the sphere $\S^{2n-1}$ into $\R^{2n}$ and the standard complexification $\R^{2n} \subset \C^{2n}$. The vector-field defined in Euclidean coordinates as $v_0(x_1,...,x_{2n}) = 2\pi(x_2,-x_1,...,x_{2n},x_{2n-1})$ defines a linear action of the circle $\S^1$ which we will denote by $\phi_t, t \in \R$, $\phi_1 = {\rm Id}$. In Euclidean coordinates
$\phi_t(x_1,...,x_{2n})=(\cos(2\pi t) x_1 +\sin(2\pi t)x_2,-\sin(2\pi t)x_1 +\cos(2\pi t)x_2,...,
\cos(2\pi t) x_{2n-1} +\sin(2\pi t) x_{2n},-\sin(2\pi t) x_{2n-1} +\cos (2\pi t)x_{2n})$.
We will use the same notations $v_0$ and $\varphi^t$ for extensions to $\C^{2n}$ or its subsets. We will call a function on $\S^{2n-1}$ entire if it extends to a holomorphic function on $\C^{2n}$. We say that the map is in $C^\omega_\Delta$ if it extends to a holomorphic function in the ball $B_\Delta:=\{z \in \C^{2n} : |z|\leq \Delta\}$. We then use the notation $h\in C^\omega_\infty$ if $h$ is entire.  A map $f : \S^{2n-1} \to \S^{2n-1}$ is said to be $C^\omega_\Delta$ if its coordinate functions are $C^\omega_\Delta$. A diffeormorphism f : $\S^{2n-1} \to \S^{2n-1}$ is $C^\omega_\Delta$ if both $f$
 and $f^{-1}$ are $C^\omega_\Delta$.  Invertible linear maps are obviously entire differomorphisms. Notice that product of entire diffeomorphisms is an entire diffeomorphism so that entire diffeomorphisms form a group that we will denote ${\rm Ent}(\S^{2n-1})$. Its subgroup of entire diffeomorphisms preserving Lebesque measure $\lambda$ is denoted by  ${\rm Ent}(\S^{2n-1},\lambda)$. A homeomorphism $h$ of a compact metric space $X$ is uniquely ergodic if it has only one invariant Borel probability measure. If h preserves a measure with full support (nonempty open sets have positive measure) then unique ergodicity implies minimality (every orbit is dense). Unique ergodicity is equivalent to a uniform distribution property: time averages of any continuous function converge uniformly to a constant which is then equal to the integral with respect to the invariant probability measure.

\subsection{Formulation of the result.}\label{results} Fix $n \in \N$. For $\Delta>0$ and $f,g \in C^\omega_\Delta (\S^{2n-1})$ we denote 
$${|f-g|}_\Delta=\max \left\{ \max_{z \in B_\Delta} |f(z)-g(z)|,  \max_{z \in B_\Delta} |f^{-1}(z)-g^{-1}(z)|\right\}$$

\begin{theo} \label{theorem.ergodic}
For any $t_0 \in [0,1]$ and any $\eps>0$, $\Delta>0$, there exists a uniquely ergodic volume preserving diffeomorphism $f \in C^\omega_\Delta (\S^{2n-1})$ that satisfies 
\begin{equation*}
{|f-\varphi^{t_0}|}_\Delta < \eps 
\end{equation*} 
Furthermore, the diffeomorphism $f$ is obtained as a limit in the $C^\omega_\Delta$ norm of entire maps of the form $F_n = H_n \circ \varphi^{t_n} \circ H_n^{-1}$,   $H_n \in {\rm Ent}(\S^{2n-1},\lambda)$.
\end{theo}
     
\begin{rema} The starting point of our argument is existence of a double transitive  family of 
 entire diffeomorphisms,   commuting with  the $S^1$ action, namely  rotations.  Our argument works whenever such a family exists with  modifications that are essentially notational. Examples are  compact  connected Like groups  and some of their homogeneous spaces. 
 
 Here are more details on the compact Lie group setting. Let $G$ be a compact connected  Lie group with probability Haar measure $\chi$.
We denote by $l_g$ and $r_g$  correspondingly the left and right translation on $G$  by the element $g\in G$.

The group $G$ can be embedded into $SO(N,\R)$ as a subgroup defined by polynomial eguations in the  matrix coefficients. 
Without loss of generality  we may assume that the image of  $G$ is Zariski dense in  $SO(N,\R)$. We consider the standard  coordinate embeddings $SO(N,\R)\hookrightarrow\R^{n^2}\hookrightarrow\C^{n^2}$.

We will call a function on $G$ {\em entire} if it extends to a holomorphic function in $\C^{n^2}$.
Entire maps and diffeormorphisms of $G$ are defined as in the previos section.
Left and right translations are given by linear maps  in matrix coordinates and thus  extend to
invertible linear maps of  $\C^{n^2}$ and are thus entire diffeomorphisms. We use the same notations for the extensions.

\end{rema}     
     
\subsection{The  approxination by conjugation construction scheme}\label{ACscheme}
We will use the {\em approxination by conjugation} method sometimes called 
Anosov--Katok method  which was originally introduced in \cite{AK}.

Without loss of generality we can assume that $t_0$ is rational,
say $t_0=P_0/Q_0$ where $P_0$ and $Q_0$ are relatively prime integers.

We will construct the desired diffeomorphism  $f$ inductively, 
as $\lim_{n\to\infty} F_n$ with $F_{-1}=\varphi^{t_0}$.
Each diffeomorphism $F_n$, $n\geq 0$, will be conjugate via an entire diffeomorphism  $H_n$
to  a rational element of the action $\varphi$ with rapidly increasing periods:
\begin{equation*}
F_n=H_n\circ \varphi^{P_{n+1}/Q_{n+1}}\circ H_n^{-1}
\end{equation*}

The conjugating  diffeomorphisms are  defined inductively as\footnote{In our case, we will actually have instead of the equality that $h_{n}\circ \varphi^{P_n/Q_n} \circ h_n^{-1}$ is close in the $C^\omega_\Delta$ norm to $\varphi^{P_n/Q_n}$ (see section \ref{intro.twist} below).}
$$H_{n}=H_{n-1}\circ h_{n}, \quad {\rm with} \quad \ h_{n}\circ \varphi^{P_n/Q_n}=\varphi^{P_n/Q_n}\circ h_{n}$$
Thus, at the $n^{\rm th}$ step of the
construction the parameters are   the diffeomorphism $h_n\in  {\rm Ent}(\S^{2n-1}, \lambda)$ and the rational $t_{n+1}=P_{n+1}/Q_{n+1}$. First one
chooses the diffeomorphism $h_n$  to make  {\it all orbits} of the $\S^1$ action $\varphi_n$
defined by 
\begin{equation}
\varphi_n = H_n\circ\varphi\circ H_n^{-1}=H_{n-1}\circ h_n\circ \varphi\circ h_n^{-1}\circ H_{n-1}^{-1}
\label{eqinductiveconjugacy}
\end{equation}
distributed in an {\it equivalent way to Lebesgue measure}, in the sense that Birkhoff averages of continuous functions along the $\varphi_n$ action become as $n$ tends to $\infty$  proportional with a fixed ratio distortion to the Lebesgue averages of these functions.  This will be sufficient to guarantee unique ergodicity of the limit map. 

Naturally, $H_n$,  although entire, is likely to have
large derivatives, and in particular to  be very large on $B_\Delta$.
Thus,   $t_{n+1}$  has to be chosen with a sufficiently large denominator $Q_{n+1}$ to make the orbits of the finite subgroup 
$$
H_n\circ\varphi^{kt_{n+1}}\circ H_n^{-1},\,\, k=0,\dots Q_{n+1}-1
$$
of the action $\varphi_n$  approximate the continuous orbits  of $\varphi_n$ sufficiently well
to maintain the uniform distribution  almost without any loss of precision. Moreover, for the convergence of the construction in the analytic norms, observe that  the $\S^1$ action $\varphi_n$  is entire (since $H_n, H_n^{-1}$ and $\varphi$ are entire), thus on every compact subset  of $\C^{2n}$,
$h_n \varphi^{t_{n+1}} h_n^{-1} \to \varphi^{t_n}$ if $t_{n+1}\to t_n$. Hence the latter further constraint on the choice of $t_{n+1}$  will 
 guarantee closeness on $B_\Delta$
between $F_{n+1}$ and $F_n$, and between their inverses.

 Since there are no other restrictions on the choice of $t_{n+1}$ the only essential 
part of the inductive step is  the consturction  of  the diffeomorphism $h_n$.  It is here where the difficuties 
of the analytic case  are  very obvious.  Since those maps are very large in the real domain
control of the complexification presents great problems. A natural approach inspired  by the  smooth case would be to   construct smooth maps first and then to make some kind of 
approximation (by polynomials or other special classes of  functions) to guarantee analyticity in a large domain. The problem however remains since even if such a general approximation procedure works the inverses would have singularities close to the real domain and the construction collapses. Thus it is necessary to find conjugating diffeomorphisms of a special form which  may be inverted  more or less explicitely  to guarantee analyticity of both maps and their inverses in
a large compelx domain. We now proceed to showing how to do this in the specific case in question.

\subsection{Making  $\S^1$ orbits uniformly distributed along a transitive torus action on the sphere} \label{intro.twist}
The action $\varphi$,  which   may of course be considered as a  subgroup of  the orthogonal group $SO(2n)$ ,has  a large centralizer in $SO(2n)$.
An easy way to see it  is to identify  $\R^{2n}$ with $\C^n$  
via the map $(x_1,x_2,\dots, x_{2n-1},x_{2n})\to (x_1+ix_2,\dots x_{2n-1}+ix_{2n})$
In the complex coordinates $\varphi$ becomes scalar action
$\varphi^t(z_1,\dots, z_n)=(\exp2\pi itz_1,\dots,\exp2\pi itz_n)$.
The unitary group $U(n)$ commutes with $\varphi$. For our puposes it is useful to notice that already the  special unitary group $SU(n)$ which has finite intersection with $\varphi$,
already acts transitively of the sphere $\S^{2n-1}$.

Assume we are given a collection of one-parameter compact subgroups of period one $k_0,\ldots,k_N$ acting transitively on $\S^{2n-1}$. Given any $t_n$ we want to construct $h_n \in \text{Ent}(\S^{2n-1},\lambda)$ and $t_{n+1}$ such that  $\bar{F}_n=h_{n} \circ \varphi^{t_{n+1}} \circ h_n^{-1}$ is arbitrarily close to $\varphi^{t_n}$ and such that the arcs of orbits of $\bar{F}_n$ of length $Q_{n+1}$ for any $x \in \S^{2n-1}$, that we denote $\cO(\bar{F}_n,Q_{n+1},x)$, are distributed with high precision in the same way as the $\T^{N+2}$ orbit $\varphi^t k_0^{s_0} \ldots k_{N+1}^{s_{N+1}} y$, $(t,s_0,\ldots,s_{N+1})\in \T^{N+2}$ for some $y \in \S^{2n-1}$ that depends on $x$. The latter distribution is equivalent to Lebesgue measure. The precision with which the orbits $\cO(\bar{F}_n,Q_{n+1},x)$ become distributed as the $\T^{N+1}$ transitive action can be made so high that even after application of the conjugacy $H_{n-1}$ it still holds that the orbits $\cO(F_n,Q_{n+1},x)$ are distributed in an equivalent way to the Lebesgue measure. 

The construction of $h_n$ and $\bar{F}_n$ is itself done using a finite number of successive conjugations of periodic times of the $\varphi$ action. This is the main ingredient in the construction and we now describe it. 

We start with $p_0/q_0=\a_0=t_n$. We consider an entire function $\psi_0$ that is constant on any $k_0$ orbit but such that $\psi_0(\varphi^t(\cdot))$ depends wildly on $t$ and $\psi_0(\varphi^{\a_0}(\cdot))=\psi_0(\cdot)$ (the translation groups $k_i$ are chosen so that such functions do exist and are simple to produce). Then if we let $g_0=k_0^{\psi_0}$  we get that $g_0 \circ \varphi^{\a_0} \circ g_0^{-1} = \varphi_{\a_0}$. As a consequence of our choices, we observe that for $\a_1=p_1/q_1$ sufficiently close to $\a_0$ we have that $f_0=g_0 \circ \varphi^{\a_1} \circ g_0^{-1} $ is close to $f_{-1}=\varphi^{\a_0}$ while due to the twisting of $\psi_0$ under the $\varphi^t$ action the orbits of $f_0$ will be distributed as the continuous $\T^2$ orbits $\varphi^s k_0^{t_0}$, $(s,t_0)\in \T^2$. 

 In the same way we introduce $g_1=k_1^{\psi_1}$ that commutes with $\varphi^{\a_1}$ and then choose $\a_2$ sufficiently close to $\a_1$ so that $f_1=g_0 g_1 \varphi^{\a_2} g_1^{-1}g_0^{-1}$ is close to $f_1$ while the orbits of $f_1$ are distributed as the $\T^3$ orbits  $\varphi^s k_0^{t_0} k_1^{t_1}$, $(s,t_0,t_1)\in \T^3$. We then follow this induction until we obtain $\a_{N+1}$ and $f_N=g_0\ldots g_N \varphi^{\a_{N+1}} g_N^{-1} \ldots g_0^{-1}$ such that $f_N$ is close to $f_{-1}=\varphi^{t_n}$ while its orbits are distributed as the $\T^{N+2}$ orbits of $\varphi^s k_0^{t_0}\ldots k_N^{t_N}$, $(s,t_0,\ldots,t_N)\in \T^{N+2}$. Thus we let $t_{n+1}=\a_{N+1}$, $h_n=g_0\ldots g_N$,  and $\bar{F}_n=f_N$ and the step $n$ construction is accomplished. 
 
 Actually,  in the above scheme we omitted an extra difficulty that is related to the control of {\it every} orbit that is necessary for unique ergodicty. Namely the points $x$ for which the orbit $\cO(f_0,x,q_1)$ is well distributed 
are those for which the $\psi_0$ twist is effective and this excludes a small  measure set of points (suppose for example that $\psi_0$ depends only on the coordinate $z_1$, then $\psi_0(\varphi^t(z))$ does not depend on $t$ for points $z$ such that $z_1=0$). To overcome this difficulty a certain number of additional conjugacies $k_{N+1},\ldots,k_{M}$ must be applied to make sure that  each point is affected by the twist in all the directions $k_0,\ldots,k_N$. A consequence of this extra difficulty is that equi-distribution of different points will happen at different times and for different indices in the maps $f_l$, $l \in [N,M]$.

 \section{Proof of Theorem \ref{theorem.ergodic}}

 \subsection{Criterion for unique ergodicity. Reduction to the main induction step.}

\begin{defin} Given $C>0$ and $\eps>0$, a finite set $\cO$ is said to be $(C,\eps)$-uniformly distributed on a manifold $X$ if for any ball $B\subset X$ of radius $\eps$ we have that $\#(\cO \cap B)/\# \cO \in (\lam(B)/C, C \lam(B))$. 
\end{defin}

\begin{defin} A finite collection of one-parameter compact subgroups of period $1$: $k_0,\ldots, k_N \in SU(n)$ is said to have a transitive action on $X$ if for all $x,y \in X$ there exists $t_0,\ldots,t_N \in [0,1)$ such that $ y=k_0^{t_0} \ldots k_N^{t_N} x.$
\end{defin}

Let $X:=\S^{2n-1}$. In the sequel we will obtain and fix a finite collection of one-parameter compact subgroups of period $1$: $k_0,\ldots, k_N \in SU(n)$ whose action is transitive on $X$.


\begin{defin} A finite set $\cO$ is said to be $\eps$-uniformly distributed along $k_0,\ldots,k_L$ and $x$ if for any ball $B$ of radius $\eps$ in $[0,1]^L$ we have that 
 $\#(\cO \cap \bar{k}^Bx)/\# \cO \in ((1-\eps) {\rm Leb} (B),(1+\eps) {\rm Leb} (B))$. We used the notation $\bar{k}^B x:= \{ y = k_0^{t_0} \ldots k_L^{t_L} x : (t_0,\ldots,t_L) \in B \}$.

\end{defin}

\begin{prop} \label{transitivity} There exists $C_0>0$ such that for any $\eps>0$, there exists $\eta>0$ such that for any $x \in X$ we have the following : If a finite set $\cO$ is $\eta$-uniformly distributed along $\bar{k}_N=k_0,\ldots,k_N$ and $x$, then   $\cO$ is $(C_0,\eps)$-uniformly distributed on $X$.
\end{prop} 

\begin{proof} The proof is straightforward by transitivity and periodicity of the action by $k_0,\ldots,k_N$ and compactness of $X$.  \end{proof}

We now state a general criterion  that we will use to  prove unique ergodicity of a volume preserving  transformations $f$. For $x \in X$ and $m \in \N$, we denote the arcs of orbits $\cO(f,M,x):=\{f^m(x) : m=1,\ldots,M \}$.

\begin{prop} \label{criterion} Let $f $ be a volume preserving homeomorphism on $X$. Assume that  there exists $C>0$ such that the following holds : For any $\eps>0$ and any $x \in X$, there exists $M \in \N$, for which $\cO(f,M,x)$ is $(C,\eps)$-uniformly distributed on $X$. Then $f$ is  uniquely ergodic.
\end{prop}

\begin{proof} The assumption implies that given any continuous function $\psi : X \to \C$, and any $x \in X$, there exists a sequence $M_n \to \infty$ such that 
$$\frac{1}{M_n} \sum_{i=1}^{M_n-1} \psi (f^i x) \in (\int_X \psi(z)d\lambda(z) /C',C' \int_X \psi(z)d\lambda(z) )$$ with $C'=2C$. It follows that all the invariant probability measures by $f$ are equivalent to Lebesgue, whence unique ergodicty.
\end{proof}

We can now state what we will request at a given step of our construction to guarantee unique ergodicity of the limiting transformation. 

\begin{prop} \label{main} If for any $t_0 \in \T$ and any $\eps >0$ and  $\Delta >0$ there exists $\tau \in \T$ and a diffeomorphism $h \in Ent(X,\lam)$ such that the entire diffeomorphism $f=h \circ \varphi^{\tau} \circ h^{-1}$ satisfies the following 
\begin{itemize} 
\item  ${|f - \varphi^{t_0}|}_\Delta < \eps$ 
\item There exists  $M \in \N$ with the property that for every $x \in X$, there exist $y \in X$ and $M'(x)< M$ such that $\cO(f,M',x)$ is $\eps$-uniformly distributed  along $k_0,\ldots,k_N$ and $y$ 
\end{itemize}
Then it is possible to construct a transformation that satisfies the requirements of Theorem \ref{theorem.ergodic}.
\end{prop}

\begin{proof} Assume that sequences $M_n\in \N$, $t_n$ and $H_n \in Ent(X,\lam)$ have been constructed such that $H_{-1}={\rm Id}$ and 
$F_n= H_n \circ \varphi^{t_{n+1}} \circ H_n^{-1}$ satisfies 
 
 \noindent {\it  $(\cH_n)$ : ${|F_n-F_{n-1}|}_\Delta<\eps/2^n$ for any $n\geq 0$; and for any $0\leq j \leq n$, and any $x \in X$, there exists $M'_j(x)<M_j$ such that $\cO(F_n,M'_j,x)$ is $(C_0,1/(j+1))$-uniformly distributed in $X$.}

 Clearly the first step $n=0$ follows from Proposition~\ref{transitivity} and the assumption. At step $n$, given $H_n$, observe that there exists $\eps_n$ such that if $h_{n+1}$ and $t_{n+2}$ and 
 $M_{n+1}$ are such that 
  $\bar{f}_{n+1}=h_{n+1} \circ \varphi^{t_{n+2}} \circ h_{n+1}^{-1}$ satisfies that  for any $x \in X$  there exists $y$ and $M'_{n+1}(x) < M_{n+1} $      such that $\cO(\bar{f}_{n+1},\bar{M}'_{n+1},x)$ is  $\eps_n$-uniformly distributed  along $k_0,\ldots,k_N$ and $y$, then, using Proposition \ref{transitivity},  $F_{n+1}= H_n \circ \bar{f}_{n+1} \circ H_n^{-1}$ satisfies that   $\cO(F_{n+1},M'_{n+1}(x),x)$ is  $(C_0,1/(n+2))$-uniformly distributed in $X$ for any $x \in X$ (we took $M'_{n+1}(x)=\bar{M}'_{n+1}(H_n^{-1}(x))<M_{n+1}$). In addition we can require due to our assumption  that $\bar{f}_{n+1}$ be arbitrarily close to $\varphi^{t_{n+1}}$. As a consequence of the latter $F_{n+1}$ will be arbitrarily close to $F_{n}$ and the requirements of 
   $(\cH_{n+1})$ will hold if we take $H_{n+1}=H_n \circ h_{n+1}$. 
  
The limiting diffeomorphism $f=\lim F_n$ thus satisfies that   $\cO(f,M'_{j}(x),x)$  is $(C_0,1/(j+1))$-uniformly distributed in $X$. The
unique ergodicity criterion of Proposition \ref{criterion}  being  satisfied by $f$, Theorem \ref{theorem.ergodic} follows.
\end{proof}

It only remains  to prove the main inductive step given by Proposition \ref{main}.  Before we do this we shall introduce now the special translations that we will use in order to move the orbits transversally to the $\varphi$-action.

\subsection{A special family of translations}\label{sec.translations} For any $q \in \N$ and $i \in \{1,\ldots,n\}$, we define $\psi_{i,q}(z)={\rm Re}(z_{i}^q)$ and $\chi_{i,q}(z)={\rm Re}((z_1-z_{i})^q)$. Clearly $\psi_{i,q}$ and $\chi_{i,q}$ are entire, since they are   polynomials in the variables $x_1,\ldots,x_{2n}$. A crucial property is that $g \circ \varphi^{p/q}= g$ for $g=\psi_{i,q}$ or $\chi_{i,q}$. 

The translations we will use are 
\begin{align*} \xi_{i}^s(z_1,\ldots,z_n)&=(z_1,\ldots,z_{i-1},e^{i2\pi s}z_i,z_{i+1},\ldots,z_n)\\
\tau_{i}^s(z_1,\ldots,z_n)&=(z_{1,s},z_2,\ldots,z_{i-1},z_{i,s},z_{i+1},\ldots,z_n) \\
 z_{1,s}&=1/2\left((e^{i2\pi s}+1)z_1+(e^{i2\pi s}-1)z_i   \right)  \\ z_{i,s}&=1/2\left((e^{i2\pi s}-1)z_1+(e^{i2\pi s}+1)z_i   \right) \end{align*}
Note that under the action by $\tau_i^s$ we have that $z_{1,s}+z_{2,s}=e^{i2\pi s}(z_1+z_2)$ while $z_{1,s}-z_{2,s}=(z_1-z_2)$. This is crucial to insure that $\left(\xi_i^{A\psi_{j,q}}\right)^{-1}= \xi_i^{-A\psi_{j,q}}$   and the similar property for $\tau_i^{A\chi_{i,q}}$.  Also, as a consequence of our definitions we have that  for any $A>0$
\begin{align*} \xi_i^{A\psi_{j,q}} \varphi_{p/q} z&= \varphi_{p/q}  \xi_i^{A\psi_{j,q}} z,  \forall j\neq i \\
 \tau_i^{A\chi_{i,q}} \varphi_{p/q} z&=\varphi_{p/q}  \tau_i^{A\chi_{i,q}}z
\end{align*}
Observe finally that $\xi_i^{A\psi_{j,q}}$ and $\xi_i^{-A\psi_{j,q}}$  are entire maps as well as $\tau_i^{A\chi_{i,q}}$
and $\tau_i^{-A\chi_{i,q}}$.

\begin{prop} \label{prop.transitive.sequence}  Let $k_0=\xi_1, k_1=\tau_2,k_2=\xi_2,k_3=\xi_3,\ldots, k_n=\xi_n$, $k_{n+1}=\tau_2,k_{n+2}=\xi_2,k_{n+3}=\tau_3,k_{n+4}=\xi_3,\ldots,k_{3n-2}=\tau_n,k_{3n-1}=\xi_n$, $k_{3n}=\tau_2, k_{3n+1}=\xi_{2},k_{3n+2}=\tau_{3}, k_{3n+3}=\xi_{3},\ldots,k_{5n-4}=\tau_n,k_{5n-4}=\xi_n$. Then the sequence $k_0,\ldots,k_{5n-4}$ is transitive.
\end{prop}

\begin{rema} The translation $k_1=\tau_2$ is not necessary in making the sequence transitive but will be useful later when we will build the conjugacy to make sure that the shear along the $z_1$ direction is triggered.
\end{rema} 

The proof of Proposition \ref{prop.transitive.sequence} will be an immediate consequence of lemma \ref{lemme.transitif.2} below.

\begin{lemm} \label{lemme.transitif.1} Fix $j=1,\ldots,n$ and $(z_1,\ldots,z_n)\in X$. For any $\rho_1,\rho_j$ such that $\rho_1^2+\rho_j^2=|z_1|^2+|z_j|^2$ there exist $t,s \in [0,1]^2$ such that $z'=\tau_j^s\xi_j^t z$ satisfies $|z'_1|=\rho_1$ and $|z'_j|=\rho_j$.
\end{lemm}

\begin{proof} Let $z_j=r_je^{i2\pi \theta_j}$. We have that
 $$z'_1=e^{i\pi s}\left(\cos(\pi s)r_1e^{i2\pi \theta_1} -i\sin(\pi s)r_2e^{i2\pi \theta_j}e^{i2\pi t}  \right)$$
hence if we choose $t+\pi/2+\theta_j=\theta_1 [2\pi]$ and $\tan(\pi s) = r_1/r_2$ we get that $z'_1=0$ hence $|z'_j|^2=r_1^2+r_j^2$. 
Since 
$$z'_j=e^{i\pi s}\left(-i\sin(\pi s)r_1e^{i2\pi \theta_1} +\cos(\pi s)r_2e^{i2\pi \theta_j}e^{i2\pi t}  \right)$$
it is also possible to choose $t$ and $s$ such that $z'_j=0$. By continuity any value between $0$ and $r_1^2+r_j^2$ is possible for $|z'_j|^2$ and the lemma is proved.
\end{proof}

\begin{lemm} \label{lemme.transitif.2} Given any $\rho_1,\ldots,\rho_n$ such that $\rho_1^2+\ldots+\rho_n^2=1$ and any $z \in X$, there exist $t_1,\ldots,t_{4n-4}$ such that 
$$z'=\tau_2^{t_{4n-4}}\xi_2^{t_{4n-5}} \ldots \tau_n^{t_{2n}}\xi_n^{t_{2n-1}} \tau_{2}^{t_{2n-2}}\xi_2^{t_{2n-3}} \ldots \tau_n^{t_2}\xi_n^{t_1} z$$ satisfies $|z'_j|=\rho_j$ for every $j=1,\ldots,n$. 
\end{lemm}
\begin{proof}
Using lemma \ref{lemme.transitif.1} repeatedly we obtain first $t_1,\ldots,t_{2n-2}$ such that  $\bar{z}=\tau_{2}^{t_{2n-2}}\xi_2^{t_{2n-3}} \ldots \tau_n^{t_2}\xi_n^{t_1} z$ satisfies $\bar{z}_j=0$ for every $j=2,\ldots,n$. Next we choose $t_{2n-1}$ and $t_{2n}$ such that  $\bar{z}^{(2)}=\tau_n^{t_{2n}}\xi_n^{t_{2n-1}}\bar{z}$ satisfies $|\bar{z}^{(2)}_n|=\rho_n$ : this is possible by lemma \ref{lemme.transitif.1} since $1=|\bar{z}_1|^2+|\bar{z}_n|^2 \geq \rho_n^2$.  
We proceed inductively so  that at each step $j\leq n-2$ we have that $|\bar{z}^{(j)}_l|=\rho_l$ for $n-j\leq l \leq n$ and   $\bar{z}^{(j)}_l=0$ for $1< l < n-j$. Indeed, since $|\bar{z}^{(j)}_1|^2+|\bar{z}^{(j)}_{n-j-1}|^2=1-\rho_n^2-\ldots-\rho_{n-j}^2\geq \rho_{n-j-1}^2$, we can apply lemma  \ref{lemme.transitif.1} and choose $t_{2n+2(j+1)-1}$ and $t_{2n+2(j+1)}$ such that  $\bar{z}^{(j+1)}=\tau_{n-j-1}^{t_{2n+2(j-2)}}\xi_{n-j-1}^{t_{2n+2(j-2)-1}}\bar{z}^{(j)}$ satisfies $|\bar{z}^{(j+1)}_{n-j-1}|=\rho_{n-j-1}$. Since $\tau_{n-j-1}$ and $\xi_{n-j-1}$ leave the $r^{\rm th}$ coordinates intact for $r\neq 1$ and $r \neq n-j-1$ we still have  $|\bar{z}^{(j+1)}_{l}|=\rho_{l}$ for $n-j\leq l\leq n$. 
\end{proof}

\begin{proof}[Proof of Proposition \ref{prop.transitive.sequence}] 

Lemma \ref{lemme.transitif.2} implies that with an adequate choice of $s_{n+1},\ldots,s_{5n-4}$ one can obtain arbitrary moduli for the coordinates of  $\bar z=k_{n+1}^{s_{n+1}} \ldots k_{5n-4}^{s_{5n-4}}z$. Next, with an adequate choice of $s_0,\ldots,s_n$, such that $s_1=0$ we can further  fix the arguments of $z'=k_{0}^{s_{0}} \ldots k_{n}^{s_{n}} \bar z=\xi_1^{s_0} \xi_2^{s_2} \xi_3^{s_3}\ldots \xi_n^{s_n} \bar z$, and the proof  of Proposition \ref{prop.transitive.sequence}
is complete.
\end{proof}

 \subsection{The inductive step of the successive conjugation construction}

The novelty in our construction is that the each step of the successive conjugacy is itself constructed through an inductive procedure that allows to saturate all the directions of the transitive sequence of rotations.

We further expand our transitive sequence of $k_i$'s by introducing  $k_{5n-3}=\tau_n,\ldots,k_{6n-5}=\tau_2$, $k_{6n-4}=\xi_2$.  

We let $M=6n-4$.

For a choice (to be determined later) of sequences $A_0,\ldots,A_M$ and $q_0,\ldots,q_M$, we let 
\begin{align*} g_0&=k_0^{A_0\psi_{2,q_0} }, \quad k_0=\xi_1 \\
g_1&=k_1^{A_1 \chi_{2,q_1}}, \quad k_1=\tau_2 \\
g_2&=  k_2^{A_2 \psi_{1,q_2}}, \quad k_2=\xi_2 \\
g_3&= k_3^{A_3 \psi_{1,q_3}}, \quad k_3=\xi_3 \\ 
&\ldots \\ 
g_n&= k_n^{A_n \psi_{1,q_n}}, \quad k_n=\xi_n \end{align*}
next, we let 
\begin{align*} 
g_{n+1}&=k_{n+1}^{A_{n+1}\chi_{2,q_{n+1}}}, \quad k_{n+1}=\tau_2 \\
g_{n+2}&= k_{n+2}^{A_{n+2}\psi_{1,q_{n+2}}}, \quad k_{n+1}=\xi_2 \\ 
&\ldots \\
g_{5n-5}&= k_{5n-5}^{A_{5n-5}\chi_{2,q_{5n-5}}}, \quad k_{5n-5}=\tau_n \\
g_{5n-4}&= k_{5n-4}^{A_{5n-4}\psi_{1,q_{5n-4}}}, \quad k_{5n-4}=\xi_n \end{align*}
and finally
\begin{align*} g_{5n-3}&=k_{5n-3}^{A_{5n-3}\chi_{2,q_{5n-3}}}, \quad k_{5n-3}=\tau_n\\
&\ldots \\
g_{6n-5}&=k_{6n-5}^{A_{6n-5}\chi_{2,q_{6n-5}}}, \quad k_{6n-5}=\tau_2\\
g_{6n-4}&=  k_{6n-4}^{A_{6n-4} \psi_{1,q_{6n-4}}},  \quad k_{6n-4}=\xi_2\end{align*}

We define for each $l\leq M$, $G_l=g_0\circ \ldots \circ g_l$.


\begin{defin} We say that $z$ is $(m;a_1,\ldots,a_s;\nu)$-transversal if for any $i \neq m$ we have  for $\lambda=0$ and  $\lambda=1$
$${\rm Leb} \{t_1,\ldots,t_s : | \lambda (a_1^{t_1}\ldots a_s^{t_s} z)_i -(a_1^{t_1}\ldots a_s^{t_s} z)_m|<\nu\}<C\nu $$
where $C$ is a constant that does not depend on $z$ or $\nu$. 
\end{defin} 

Notice that if $z$ is such that $|z_1|>\eta$ then for any $\nu>0$ sufficiently small  
$${\rm Leb} \{t : |  (\xi_j^t z)_j -z_1|<\nu\}<C\nu,$$
a property that we denote by $z$ is $((1,j),\xi_j;\nu)$-transversal.

\begin{prop} \label{prop.explicit} Given any $\a_0=p_0/q_0 \in [0,1)$, any $\eta>0$, $\Delta>0$ and $\eps>0$, there exist sequences $A_0,\ldots,A_M$,  $\a_1=p_1/q_1,\ldots,\a_{M+1}=p_{M+1}/q_{M+1}$,  and $\eps^{100}=\eps_0\geq\eps_1\geq \ldots \geq \eps_M$ such that if we denote $f_l=G_{l} \varphi^{\a_{l+1}} G_{l}^{-1}$, $f_{-1}=\varphi^{\a_0}$  
we have  
\begin{enumerate} 
\item ${\left| f_{l+1}^i- f_{l}^i  \right|}_{\Delta} <\eps_0/2^{l+2}, \quad \forall |i|\leq q_{l+1}$ and  $M-1\geq l\geq -1$. 

\item For any $M>l>2$ and any $M\geq L\geq l$, and any  $\cO$ such that $\cO$ is $\eps_l$-uniformly distributed along $\varphi$,$k_l,\ldots, k_L, z$ and if $z$ is $(1;k_l,\ldots, k_L;\eps_l)$-transversal then $g_{l-1}\cO$ is  $\eps_{l-1}$-uniformly distributed along $\varphi$,$k_{l-1},\ldots, k_L, z$ and $z$ is $(1;k_{l-1},k_l,\ldots, k_L;\eps_{l-1})$-transversal.

\item (Case $l=2$). For any  $M\geq L\geq 2$, and any  $\cO$ such that $\cO$ is $\eps_2$-uniformly distributed along $\varphi$, $k_2,\ldots, k_L, z$ and if $z$ is $(1;k_2,\ldots, k_L;\eps_2)$-transversal then $g_{1}\cO$ is  $\eps_{1}$-uniformly distributed along $\varphi$,$k_{1},\ldots, k_L, z$ and $z$ is $(2;k_{1},k_2,\ldots, k_L;\eps_{1})$-transversal  (the difference form the previous property is in the change of the transverse coordinate from 1 to 2)).

\item (Case $l=1$). If $\cO$ is  $\eps_{1}$-uniformly distributed along $\varphi$,$k_{1},\ldots, k_L, z$ and $z$ is $(2;k_{1},k_2,\ldots, k_L;\eps_{1})$-transversal, then  $g_0 \cO$ is $\eps_0$-uniformly distributed along $\varphi$,$k_{0},\ldots, k_L, z$.

\item If $|z_1|>\eta $ then  $g_{6n-4}(\cO(\varphi^{\a_{6n-3}},q_{6n-3},z))$ is $\eps_{6n-4}$-uniformly distributed  along $\varphi, k_{6n-4}$ and $z$; and $z$ is $((1,2); k_{6n-4};\eps_{6n-4})$-transversal.

\item  If $\cO$ is $\eps_{6n-4}$-uniformly distributed along $\varphi, k_{6n-4}$ and $z$ and if 
$z$ is $((1,2),k_{6n-4};\eps_{6n-4})$-transversal then $g_{6n-5} \cO$ is  $\eps_{6n-5}$-uniformly distributed along $\varphi,k_{6n-5},k_{6n-4}$ and $z$ is $(1,k_{6n-5},k_{6n-4};\eps_{6n-5})$-transversal.

\item If for some $n\geq j\geq 2$ we have that $|z_1-z_j|>\eta$ then $g_{6n-j-3}(\cO(\varphi^{\a_{6n-j-2}},q_{6n-j-2},z))$ is $\eps_{6n-j-3}$-uniformly distributed  along $k_{6n-j-3}$ and $z$; and $z$ is $(1; k_{6n-j-3};\eps_{6n-j-3})$-transversal.

\end{enumerate} 

\end{prop}

\begin{proof}[Proof of Theorem \ref{theorem.ergodic}] 
Before we prove Proposition \ref{prop.explicit} we show how it implies that $f_M=G_M \varphi^{\a_{M+1}} G_M^{-1}$ satisfies the requirements of Proposition \ref{main}, from where  Theorem \ref{theorem.ergodic} would follow.

First of all, it follows from 1) of Proposition \ref{prop.explicit} and a choice of $\a_0$ such that $|\a_0-t_0|<\eps_0/2$ that 
${\left| f_M- \varphi^{t_0}  \right|}_{\Delta} <\eps_0$. 

Given any $x \in X$, we claim that there exists $N\leq l\leq M$ such that $\cO(f_l,q_{l+1},x)$ is $\eps_0$-uniformly distributed  along $\{k_0,\ldots,k_l\}$ and some $y \in X$. Due to 1) of  Proposition \ref{prop.explicit} this is sufficient to yield a similar property for $f_M$ (if we replace $\eps_0$ by $\eps$). But uniform distribution along $\{k_0,\ldots,k_l\}$  yields {\it a fortiori} uniform distribution along $\{k_0,\ldots,k_N\}$ and hence the requirements of Proposition \ref{main} will be satisfied.

To prove our claim, we first need the following immediate lemma.

\begin{lemm} \label{lemma.start} Define $\eta=1/4^{n+1}$. Then, given any $\bar z \in X$, then either $|{\bar z}_1|>\eta$ or there exists $j\in[2,n]$ such that $z=g_{6n-j-2}\ldots g_{6n-4} \bar z$ satisfies $|z_1-z_j|> \eta$. 
\end{lemm}

\begin{proof}[Proof of lemma \ref{lemma.start}]  Define $\eta_j=4^j/ 4^{n+1}$ for $j=1,\ldots,n$.  If $|\bar{z}_1|\leq \eta_1$ while $|\bar{z}_2|\geq \eta_2$ then since  $z=g_{6n-4} \bar z$ satisfies $z_1=\bar{z}_1$ and $|z_2|=|\bar{z}_2|$ we get that $|z_1-z_2|\geq \eta_2-\eta_1\geq 2\eta_1.$ We now apply a similar argument for $j>2$.

We first show by induction on $j$ that if $|\bar{z}_i|\leq \eta_i$ for $i=1,\ldots,j$ then $z'=\tau_{j}^{t_j}\ldots \tau_{2}^{t_2} \xi_2^{t_1} \bar{z}$ satisfies  $|z'_1|\leq 2\eta_{j}$  for any choice of $t_1,\ldots,t_j$. Indeed if we suppose the latter true up to $j$ and assume in addition that  $|\bar{z}_{j+1}|\leq \eta_{j+1}$ then $z:=\tau_{j+1}^{t_{j+1}} z'$ satisfies $z_{1}=1/2\left((e^{i2\pi t_{j+1}}+1)z'_1+(e^{i2\pi t_{j+1}}-1) \bar{z}_{j+1}   \right) \leq 2\eta_j+\eta_{j+1} \leq 2 \eta_{j+1}$ (we used that $z'_{j+1}=\bar{z}_{j+1}$).

Now, if to the contrary we suppose that $|\bar{z}_i|\leq \eta_i$ for $i=1,\ldots,j$ while $|\bar{z}_{j+1}|\geq \eta_{j+1}$ and use the same notations as above for $z'$ and $z$ then since $z_{j+1}-z_1=z'_{j+1}-z'_1=\bar{z}_{j+1}-z'_1$ we get that 
$\|z_{j+1}-z_1|\geq \eta_{j+1}-2\eta_j = 2\eta_j$. 

But since $\sum_{i=1}^n |\bar{z}|^2 =1$, we have that if $|{\bar z}_1|\leq \eta$, then there necessarily exists a $j\in [1,n-1]$ such that   $|\bar{z}_i|\leq \eta_i$ for $i=1,\ldots,j$ while $|\bar{z}_{j+1}|\geq \eta_{j+1}$. This finishes the proof of the lemma. \end{proof}



Back to the requirements of Proposition \ref{main}, given $x \in X$ we let $\bar z=G_{6n-4}^{-1}x$. Then we have two alternatives

$\bullet$ If $|\bar z_1|> \eta$, we prove that 
$\cO(f_{6n-4},q_{6n-3},x)$ is $\eps$-uniformly distributed  along $\{k_0,\ldots,k_{6n-4}\}$ and  $\bar z$. Indeed, this amounts to proving that  \newline $G_{6n-4} 
(\cO(\varphi^{\a_{6n-3}},q_{6n-3},\bar z))$ is $\eps_0$-uniformly distributed  along $\{k_0,\ldots,k_{6n-4}\}$ and $\bar z$. To obtain the latter, we apply 5) of Proposition  \ref{prop.explicit}, then 6), and then 2) inductively until we finish with 3) then 4). 

$\bullet$ If $|\bar z_1|\leq \eta$, then for $j$ as in Lemma \ref{lemma.start} we let 
$z=g_{6n-j-2}\ldots g_{6n-4} \bar z= G_{6n-j-3}^{-1}x$, and we prove that  
 $\cO(f_{6n-j-3},q_{6n-j-2},x)$ is $\eps_0$-uniformly distributed  along $\{k_0,\ldots,k_{6n-j-3}\}$ and  $z$. Indeed, it is sufficient to prove that 
\newline $G_{6n-j-3} 
(\cO(\varphi^{\a_{6n-j-2}},q_{6n-j-2},z))$ is $\eps_0$-uniformly distributed  along $\{k_0,\ldots,k_{6n-j-3}\}$ and $ z$. Since $|z_1-z_j|>\eta$, we can apply 6) of Proposition  \ref{prop.explicit} and then 2) inductively until we finish with 3) then 4). 

The proof of Theorem \ref{theorem.ergodic} is hence completed. 
\end{proof}

\begin{proof}[Proof of Proposition \ref{prop.explicit}]

Proposition \ref{prop.explicit} is proved by a finite induction of which the main building block is provided by the following straightforward fact:

{\it For any $\eps>0$ and any $a_1,\ldots,a_s \in \{k_0,\ldots,k_M\}$ there exists $A>0$ and $\eps'>0$ such that: if $g=a^{A \upsilon_{1,q}}$ with $q\geq Q$ and $(a,\upsilon)=(\tau_j,\chi)$ or $(a,\upsilon)=(\xi_j,\psi)$ and if 
$\cO$  is $\eps'$-uniformly distributed along $\varphi$,$a_1,\ldots, a_s, z$ and if $z$ is $(1;a_1,\ldots, a_s;{\eps'})$-transversal then $g\cO$ is  $\eps$-uniformly distributed along $\varphi$,$a,a_1,\ldots, a_s, z$ and $z$ is $(1;a,a_1,\ldots, a_s;\eps)$-transversal.}

The latter as we will see will be useful for the proof of 2) of Proposition \ref{prop.explicit}. Similar statements are valid and serve for proving 3)--7) of the proposition.



We describe now how the finite induction is carried out to prove Proposition \ref{prop.explicit}. Assume that we are given $A_i$ for $i\leq 6n-5$ and $\a_i, \eps_i$ for $i\leq 6n-4$. Then we choose $A_{6n-4}$ sufficiently large and $\eps'$ such that if $\cO$ is $\eps'$-uniformly distributed along $\varphi$ and $z$ and if $|z_1|>\eta$ then $g_{6n-4} \cO$ is $\eps_{6n-4}$-uniformly distributed along $\varphi,k_{6n-4}$ and $z$ is $((1,2),k_{6n-4};\eps_{6n-4})$-transversal. 

Now, we choose $\a_{6n-3}$ such that 1) of proposition \ref{prop.explicit} holds with $l=6n-5$ and $\cO(\varphi^{6n-3},q_{6n-3},z)$ is $\eps'$-uniformly distributed along $\varphi$ and $z$. Hence (5) of proposition \ref{prop.explicit} holds.

Next, given $A_i$ for $i\leq 6n-6$ and $\a_i, \eps_i$ for $i\leq 6n-5$, we choose $A_{6n-5}$ sufficiently large and $\eps_{6n-4}$ such that if $\cO$ is $\eps_{6n-4}$-uniformly distributed along $\varphi, k_{6n-4}$ and $z$ and if 
$z$ is $((1,2),k_{6n-4};\eps_{6n-4})$-transversal then $g_{6n-5} \cO$ is  $\eps_{6n-5}$-uniformly distributed along $\varphi,k_{6n-5},k_{6n-4}$ and $z$ is $(1,k_{6n-5},k_{6n-4};\eps_{6n-5})$-transversal.  
Then we choose $\a_{6n-4}$ to guarantee $1)$ of proposition \ref{prop.explicit} with $l=6n-6$. 
 We can also ask from our choice of $A_{6n-5}$ and $\eps_{6n-4}$ and $\a_{6n-4}$ that (7) of proposition \ref{prop.explicit} holds.

We can continue inductively : for $l$ decreasing from $l=6n-5$ to $l=3$, we assume given $A_i$ for $i\leq l-2$ and $\a_i, \eps_i$ for $i\leq l-1$, we choose $A_{l-1}$ and $\eps_{l}$  such that $2)$ of 
proposition \ref{prop.explicit} holds, then we choose $\a_l$ such that $1)$ of proposition \ref{prop.explicit} holds, that is  
 ${\left| f_{l-1}^i- f_{l-2}^i  \right|}_{\Delta} <\eps_0/2^{l}, \quad \forall |i|\leq q_{l}$.  For $l\geq 5n-2$ we also ask that  (7) of proposition \ref{prop.explicit} holds. 

 For $l=2$, we choose $A_1$ and $\eps_2$, then $\a_2$ such that $3)$ of proposition \ref{prop.explicit} holds and 
 ${\left| f_{1}^i- f_{0}^i  \right|}_{\Delta} <\eps_0/4, \quad \forall |i|\leq q_{2}$. To finish, we choose 
 $A_0$ and $\eps_1$, then $\a_1$ such that $4)$ of proposition \ref{prop.explicit} holds and 
 ${\left| f_{0}^i- f_{-1}^i  \right|}_{\Delta} <\eps_0/2, \quad \forall |i|\leq q_{1}$.

\end{proof}

\end{document}